\pdfoutput=1

\documentclass[11pt]{amsart}

\usepackage{mathtext}
\usepackage[T2A]{fontenc}
\usepackage[russian,english]{babel}
\usepackage[margin=1.1in]{geometry}
\usepackage{graphicx}
\usepackage{amsfonts,amssymb,amscd,amsmath,amsthm}
\usepackage{mathrsfs}
\usepackage{epsf}
\usepackage{wrapfig}
\usepackage{cite}
\usepackage[hyper]{amsbib}

\newcommand{\R}{\mathbb{R}}
\DeclareMathOperator \conv {conv} 
\DeclareMathOperator \dmt {dmt}

\begin{document}

\theoremstyle{plain}
\newtheorem{theorem}{Theorem}
\newtheorem{lemma}{Lemma}
\newtheorem{conjecture}{Conjecture}

\title{Polyhedral characteristics of balanced and~unbalanced~bipartite subgraph problems}

\author{Vladimir Bondarenko, Andrei Nikolaev, Dzhambolet Shovgenov}
\thanks {The research was partially supported by the Russian Foundation for Basic Research, Project 14-01-00333, the President of Russian Federation Grant MK-5400.2015.1, and the initiative R\&D VIP-004 AAAA-A16-116070610022-6.}

\address{%
Department of Discrete Analysis, P.G. Demidov Yaroslavl State University, Sovetskaya, 14, Yaroslavl, 150000, Russia
}
\email {bond@bond.edu.yar.ru, andrei.v.nikolaev@gmail.com, djsh92@mail.ru}

\begin{abstract}
We study the polyhedral properties of three problems of constructing an optimal complete bipartite subgraph (a biclique) in a bipartite graph.
In the first problem we consider a balanced biclique with the same number of vertices in both parts and arbitrary edge weights.
In the other two problems we are dealing with unbalanced subgraphs of maximum and minimum weight with nonnegative edges.
All three problems are established to be NP-hard.
We study the polytopes and the cone decompositions of these problems and their 1-skeletons. 
We describe the adjacency criterion in 1-skeleton of the polytope of the balanced complete bipartite subgraph problem. 
The clique number of 1-skeleton is estimated from below by a superpolynomial function.
For both unbalanced biclique problems we establish the superpolynomial lower bounds on the clique numbers of the graphs of nonnegative cone decompositions. 
These values characterize the time complexity in a broad class of algorithms based on linear comparisons.
\end {abstract}

\keywords{Biclique, 1-skeleton, cone decomposition, clique number, NP-hard problem.}

\maketitle


\section* {Introduction}

We consider the well-known balanced complete bipartite subgraph problem.

\textbf{Balanced complete bipartite subgraph (BCBS).}
Given a bipartite graph $G=(U,V,E)$ and a positive integer $k \leq |U|+|V|$. 
Are there two disjoint subsets $U_{x} \subseteq U$ and $V_{x} \subseteq V$ such that $|U_{x}| = |V_{x}| = k$ and $\{u,v\} \in E$ for any pair of vertices $u \in U_{x}$, $v \in V_{x}$.

The problem is NP-complete by reduction from Clique \cite {Garey, Johnson}.
The optimization version of the problem where it is required to find a balanced complete bipartite subgraph with a maximum number of vertices is extremely hard even to approximate (no constant factor approximation algorithms are known \cite {Feige}).
A significant number of works are devoted to the study of polynomially solvable special cases of the problem \cite {Arbib, Mubayi}.
Note that a complete bipartite graph is also called a biclique \cite {Diestel}.

It is interesting that the related unbalanced complete bipartite subgraph problem (if we replace the condition $|U_{x}| = |V_{x}| = k$ by $|U_{x}| + |V_{x}| = k$) is polynomially solvable.
The algorithm is based on the well-known K\"{o}nig's theorem \cite{Diestel}.

\begin{theorem} [K\"{o}nig]
In any bipartite graph, the number of edges in a maximum matching equals the number of vertices in a minimum vertex cover.
\end {theorem}
	
A maximum matching in a bipartite graph can be found in time $O(|E| \sqrt{|U|+|V|})$ by Hopcroft-Karp algorithm \cite {Hopcroft}.
The proof of K\"{o}nig's theorem is constructive, on its basis one can find a vertex cover.
All vertices that are not included in a minimum vertex cover form a maximum independent set.
If we consider a complement of a bipartite graph, then an independent set is transformed into a complete bipartite subgraph.
	
The problems related to the construction of complete bipartite subgraphs often arise in various applied fields.
In particular, in computational biology for biclustering of expression data of genes a bipartite graph is constructed in which one side represent genes and the other side their properties. The goal is to find a maximum balanced biclique in such graph \cite{Cheng}.
The BCBS problem is also applied in VLSI theory for PLA-folding to reduce the size of programmable logical arrays (PLA) \cite{Ravi}.
	
We consider three optimization versions of the biclique problem in a weighted bipartite graph.
	
\textbf{Weighted balanced complete bipartite subgraph (WBCBS).}
Given a complete bipartite graph $G=(U,V,E)$, $|U|=|V|=n$, a weight function $C: E \rightarrow \R$, and a positive integer $k \leq n$. 
It is required to find a balanced complete bipartite subgraph $G_{x} = (U_{x},V_{x},E_{x})$ with the largest (least) sum of edge weights such that $|U_{x}| = |V_{x}| = k$.
	
\textbf{Maximum weighted complete bipartite subgraph (maxWCBS).}
Given a complete bipartite graph $G=(U,V,E)$, $|U|=|V|=n$, a weight function $C: E \rightarrow \R_{+}$, and a positive integer $k \leq 2n$.
It is required to find a complete bipartite subgraph $G_{x} = (U_{x},V_{x},E_{x})$ with the sum of edge weights being as large as possible such that $|U_{x}| + |V_{x}| = k$.
	
\textbf{Minimum weighted complete bipartite subgraph (minWCBS).}
Given a complete bipartite graph $G=(U,V,E)$, $|U|=|V|=n$, a weight function $C: E \rightarrow \R_{+}$, and a positive integer $k \leq 2n$.
It is required to find a complete bipartite subgraph $G_{x} = (U_{x},V_{x},E_{x})$ with the sum of edge weights being as small as possible such that $|U_{x}| + |V_{x}| = k$.
	
We note that for the problem of a balanced biclique, the question of whether the problem is on minimum or maximum is not crucial. 
In both cases, the solution is a $k$-balanced complete bipartite subgraph.
Further, we only consider the maximum problem. The minimum problem can be obtained from it by inverting the sign of edge weights.
	
At the same time, for an unbalanced case, the minimum and maximum problems are fundamentally different.
Indeed, if we consider a graph $G$ with identical positive edge weights (all missing edges have a zero weight for the maximum problem, or an infinity weight for the minimum problem, thus we can assume that the graph is complete), then the maximum is achieved on a balanced or almost balanced biclique with the largest number of edges, while for the minimum problem biclique has to be as unbalanced as possible.
In order to take this into account, we consider these problems with nonnegative edge weights.

\section{Cone decompositions}

The object of the research is the construction of cone decomposition.
Let $X$ be a finite set of points of $\R^{d}$.
We consider the problem of maximizing a linear objective function over $X$:
$$\left\langle c,x\right\rangle \rightarrow \max, \ x \in X.$$
We denote by
\begin {equation}
K(x) = \{c\in \R^{d}:\ \left\langle c,x\right\rangle\geq \left\langle c,y\right\rangle,\ \forall y \in X\}.
\label{ConeDefinition}
\end {equation}
Since $K(x)$ is the set of solutions of a finite system of homogeneous linear inequalities, it is a convex polyhedral cone.
	
Given that
$$\bigcup_{x\in X} K(x) = \R^{d},$$
the set of all cones $K(x)$ is called the \textit {cone decomposition} of the space $\R^{d}$ by the set $X$.
The cone decomposition is similar to the Voronoi diagram, exactly coinciding with it, if the Euclidean norms of all points of the set $X$ are equal to each other.
	
We consider the graph of the cone decomposition with the cones being the vertices, and two cones $K(x)$ and $K(y)$ are adjacent if and only if they have a common facet:
$$\dim (K(x) \cap K(y)) = d - 1.$$
	
We denote by $\omega (X)$ the clique number, the number of vertices in a maximum clique, of the graph of the cone decomposition of the space $\R^{d}$ by the set $X$. It is known \cite {Bondarenko-Maksimenko} that the complexity of the direct type algorithms, based on linear comparisons, of finding the minimum (or maximum, if we change the sign of inequality in the definition of cone) of a linear objective function over the set $X$, or, in other words, finding the cone $K(x)$ with the vector $c$, cannot be less than the value of $\omega (X)-1$.
Thus, $\omega (X)$ characterizes the time complexity in a broad class of algorithms.
	
Let $M(X)$ be the convex hull of $X$: $M(X) = \conv (X)$.
The convex hull of a finite set of points is a convex polytope that is called the polytope of the problem.
	
We note that the following lemma holds for the cone decomposition over the entire space $\R^{d}$ (\ref{ConeDefinition}) \cite {Bondarenko-Maksimenko}.
	
\begin{lemma} \label{Cones-vertices}
Vertices $x$ and $y$ of the polytope $M(X)$ are adjacent if and only if the cones $K(x)$ and $K(y)$ have a common facet.
\end{lemma}
	
Thus, for the cone decomposition of the space $\R^{d}$ by the set $X$, the graph coincides with 1-skeleton of the polytope $M(X)$. The skeleton of the polytope $M(X)$ (also called $1$-skeleton) is the graph whose vertex set is the vertex set of $M(X)$ (in this case it is $X$) and edge set is the set of 1-faces of $M(X)$.
	
Similarly, we construct the cone decompositions of the positive orthant $\R^{d}_{+}$ for problems on maximum and minimum:
\begin {gather}
K^{+}_{\max}(x) = \{c\in \R^{d}_{+}:\ \left\langle c,x\right\rangle \geq \left\langle c,y\right\rangle,\ \forall y \in X\},
\label{ConeMax}\\
K^{+}_{\min}(x) = \{c\in \R^{d}_{+}:\ \left\langle c,x\right\rangle \leq \left\langle c,y\right\rangle,\ \forall y \in X\}.
\label{ConeMin}
\end {gather}
This construction, in turn, is dual to the polyhedron of the problem that is defined as the dominant of the convex hull of the set $X$:
$$\dmt (X) = \conv (V) + \R^{d}_{+},$$
and is used in the analysis of problems with nonnegative input data \cite {Grotschel}.
In our case it is the nonnegative weights of edges.
	
A large number of papers have been devoted to the study of 1-skeletons, graphs of cone decompositions, and their interrelation with the complexity of combinatorial optimization problems.
In particular, there are results for 1-skeletons of the traveling salesperson problem \cite{Bondarenko-TSP}, spanning tree problems with additional constraints \cite{Bondarenko-Shovgenov}, nonnegative cone decompositions for the shortest and longest path problems \cite{Maksimenko}, maximum and minimum cut problems \cite {Bondarenko-Nikolaev-13, Bondarenko-Nikolaev-16}, and many others \cite {Bondarenko-Maksimenko}.

\section{Balanced complete bipartite subgraph}
	
\begin{theorem}
The WBCBS problem is NP-hard.
\label{balanced-NPhard}
\end{theorem}
	
\begin{proof}
We consider only the maximum problem. The proof for the case on minimum is completely similar.	
		
Let $G=(U,V,E)$ be a bipartite graph with $|U|=|V|=n$.
We construct a complete bipartite graph $G^{*} = (U,V,E^{*})$ with the weight function:
\begin{gather*}
c(i,j) =
\begin{cases}
1,& \mbox {if } (i,j) \in E,\\
0,& \mbox {otherwise.}
\end{cases}
\end{gather*}
		
There exists a biclique $x$ such that $|U_{x}| = |V_{x}| = k$ in the graph $G^{*}$ with the sum of edge weights equal to $k^{2}$ (this is the maximum possible weight of a $k$-balanced complete bipartite subgraph) if and only if $x$ is a biclique in the graph $G$. NP-complete problem BCBS is polynomially reduced to WBCBS, respectively, the latter is NP-hard.
\end{proof}
	
\textit {Note:} for the sake of convenience, we use the weights of edges $-1$, $0$, $1$, and $+\infty$.
Since the number of edges in a complete graph is finite, these weights can be replaced by positive integer numbers $\{a_1,a_2,a_3,a_4\}$ such that for any $i$ the weight $a_{i}$ is strictly greater than the sum of weights of all edges $a_{i-1}$.
	
With each feasible solution $x$ of the WBCBS problem, with each $k$-balanced subgraph of $G$, we associate a characteristic vector of the space $\R^{n^{2}}$ by the following rule:
\begin{gather*}
x_{i,j} = 
\begin{cases}
1,& \mbox {if } i \in U(x),\ j \in V(x),\\
0,& \mbox {otherwise.}
\end{cases}
\end{gather*}
	
We denote by $X_{n,k}$ the set of characteristic vectors of all feasible solutions. We consider the polytope of the weighted balanced complete bipartite subgraph problem $WBCBS(n,k) = \conv (X_{n,k})$ and the cone decomposition $K_{n,k}$ of the space $\R^{n^{2}}$ by the set $X_{n,k}$. Let $c \in \R^{n^{2}}$ be the vector of edges weights of the graph $G$, then the sum of edge weights of the subgraph $G_{x}$ is equal to the value of the objective function $\left\langle c,x\right\rangle$.
	
\begin{lemma} \label{balanced-adjacency}
Vertices $x$ and $y$ of the polytope $WBCBS(n,k)$ are adjacent if and only if the corresponding bipartite subgraphs have no common parts:
$$U(x) \neq U(y) \mbox { and } V(x) \neq V(y),$$
or the subgraphs coincide in one part, and in the other part they differ exactly in one vertex:
\begin{gather*}
\left[
\begin{array} {l}
U(x) = U(y),\ |V(x) \backslash V(y)| = 1,\\
V(x) = V(y),\ |U(x) \backslash U(y)| = 1.
\end{array}
\right.
\end{gather*}
\end{lemma}

\begin{proof}
By Lemma \ref{Cones-vertices}, the adjacency of vertices of the polytope $WBCBS(n,k)$ is equivalent to the adjacency of cones of the partition $K_{n,k}$.
We carry out the proof from the point of view of a cone decomposition.
		
Let $x,y \in X_{n,k}$.
Since the cones $K_{n,k}(x)$ and $K_{n,k}(y)$ are adjacent, there is a vector $c \in \R^{n^2}$ that belongs only to these two cones of the cone decomposition $K_{n,k}$ (cones $K_{n,k}(x)$ and $K_{n,k}(y)$ have a common facet)
\begin{equation} \label {adjacency-cones}
\exists c \in \R^{n^2}, \forall z \in X_{n,k} \backslash \{x,y\}: \left\langle c,x\right\rangle = \left\langle c,y\right\rangle > \left\langle c,z\right\rangle.
\end{equation}
		
Let the subgraphs $x$ and $y$ have no common parts. We construct the vector $c$ of edge weights by the following rule (Fig. \ref {Fig_1}):
\begin{gather} 
c_{i,j} =
\begin{cases}
1,& \mbox {if } i \in U(x), j \in V(x) \mbox { or } i \in U(y), j \in V(y), \\ 
0,& \mbox {otherwise.}
\end{cases}
\label {0/1-weight}
\end{gather}
		
\begin{figure}[h]
\centering
\includegraphics[width=2.5in]{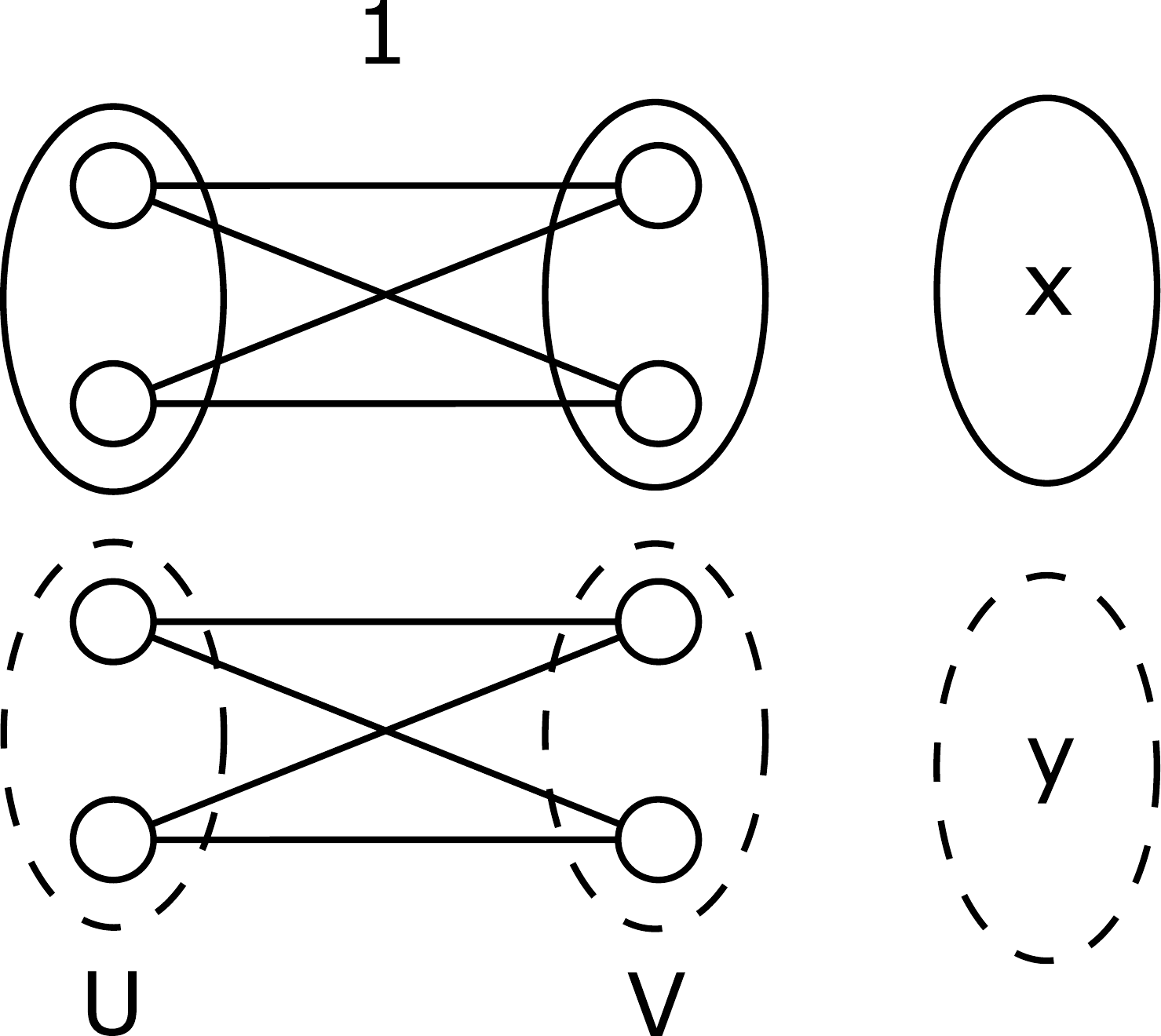}
\caption {The weight function for balanced subgraphs without common parts}
\label{Fig_1}
\end{figure}
		
In this case we obtain
$$\left\langle c,x\right\rangle = \left\langle c,y\right\rangle = k^2,$$
and this is the maximum possible weight of $k$-balanced biclique in the graph.
		
We consider an arbitrary subgraph $z \in X_{n,k} \backslash \{x,y\}$.
\begin {itemize}
\item If $z$ has at least one vertex in the part $U$ that does not belong to $U(x) \cup U(y)$, then $\left\langle c,z\right\rangle < k^{2}$, since all edges incident to this vertex have zero weight.
\item If $z$ has vertices both of $U(x) \backslash U(y)$ and $U(y) \backslash U(x)$ in the part $U$, then in the right part $V$ only vertices of $V(x) \cap V(y) $ have nonzero edges with them at the same time. Since $|V (x) \cap V(y)| < k$, at least one edge has zero weight, hence $\left\langle c,z\right\rangle < k^{2}$.
\item If $U(z) = U(x)$ (or $U(z) = U(y)$), then we can similarly consider the part $V$ and show that since $z \neq x$ ($z \neq y$, respectively), we have $\left\langle c,z\right\rangle < k^{2}$.
\end {itemize}
		
Thus, the cones $K_{n,k} (x)$ and $K_{n,k}(y)$ are adjacent by condition (\ref {adjacency-cones}).
		
Let the subgraphs $x$ and $y$ have a common part. 
Without loss of generality, we assume that $V(x) = V(y)$.
Let $|U(x) \backslash U(y)| = 1$.
We construct the vector $c$ of edge weights by the following rule (Fig. \ref{Fig_2}):
\begin{gather*}
c_{i,j} = 
\begin{cases}
1,& \mbox {if } i \in U(x) \cap U(y),\ j \in V(x) = V(y), \\ 
0,& \mbox {if } i \in U(x) \triangle U(y),\ j \in V(x) = V(y), \\
-1,& \mbox {otherwise.}
\end{cases}
\end{gather*}
		
\begin{figure}[h]
\centering
\includegraphics[width=2.5in]{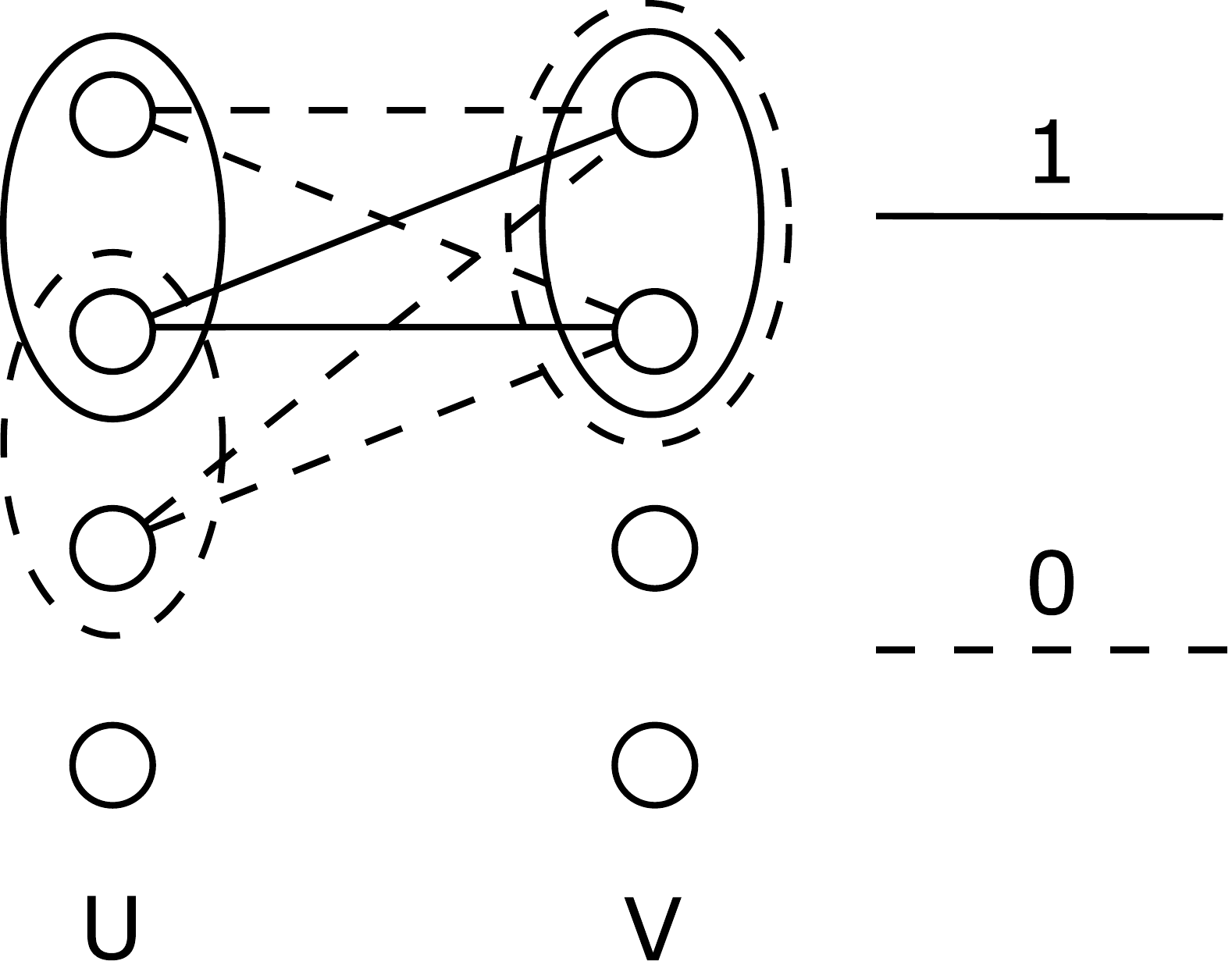}
\caption {The weight function for balanced subgraphs with a common part}
\label {Fig_2}
\end{figure}
		
By construction, we obtain
$$\left\langle c,x\right\rangle = \left\langle c,y\right\rangle = k(k-1).$$
		
We consider an arbitrary subgraph $z \in X_{n,k} \backslash \{x,y\}$.
\begin {itemize}
\item If $z$ has at least one vertex in the part $U$ that does not belong to $U(x) \cup U(y)$, then all edges incident to this vertex have negative weights, hence 
$$\left\langle c,z\right\rangle \leq k(k-1) - k < k(k-1).$$
\item If $z$ has both vertices of $U(x) \triangle U(y)$, then the edges incident to them have zero or negative weight, hence
$$\left\langle c,z\right\rangle \leq k(k-2) < k(k-1).$$
\end {itemize}
Consequently, the cones $K_{n,k} (x)$ and $K_{n,k}(y)$ are adjacent.
		
It remains to consider the case $|U(x) \backslash U(y)| \geq 2$.
We denote by
$$a = |U(x) \cap U(y)| \leq k-2.$$
The symmetric difference $U(x) \triangle U(y)$ contains $2(k-a)$ vertices.
We can choose $k-a$ vertices from them in 
$${2(k-a) \choose k-a} \geq 6$$
different ways. Therefore, there exist subgraphs $z,t \in X_{n,k} \backslash \{x,y\}$ such that
\begin{gather*}
V(z) = V(t) = V(x) = V(y),\\
U(z) \cap U(t) = U(x) \cap U(y),\\
U(z) \cup U(t) = U(x) \cup U(y).
\end{gather*}
Thus,
$$\left\langle c,x\right\rangle + \left\langle c,y\right\rangle = \left\langle c,x \cup y\right\rangle + \left\langle c,x \cap y\right\rangle = \left\langle c,z \cup t\right\rangle + \left\langle c,z \cap t\right\rangle = \left\langle c,z\right\rangle + \left\langle c,t\right\rangle,$$ 
and at least one of the subgraphs $z$ or $t$ has a weight as large as $x$ and $y$.
The cones $K_{n,k} (x)$ and $K_{n,k}(y)$ are not adjacent.
\end{proof}
	
\begin {theorem} \label {balanced-clique}
The clique number of 1-skeleton of the polytope of the weighted balanced complete bipartite subgraph problem $WBCBS(n,k)$ is superpolynomial in $n$ and $k$:
$$\omega (WBCBS(n,k)) \geq {n \choose k} = \Omega \left(\left(\frac{n}{k}\right)^{k}\right).$$
\end {theorem}
	
\begin {proof}
We consider a subset of $k$-balanced bipartite subgraphs $Y_{n,k} \subset X_{n,k}$ of the following form: we number all vertices in each part by numbers from $1$ to $n$ and consider only subgraphs whose vertex numbers on the left and right sides coincide.
Any two subgraphs $x,y \in Y_{n,k}$ have no common parts. Hence, by Lemma \ref{balanced-adjacency}, the corresponding vertices of the polytope $WBCBS(n,k)$ are adjacent.
Thus, $Y_{n,k}$ forms a clique in 1-skeleton of the polytope of the problem:
$$|Y_{n,k}| = {n \choose k}.$$
The asymptotic lower estimate is standard for the number of combinations.
\end {proof}

\section {Maximum complete bipartite subgraph}
	
Now we consider the unbalanced case.
	
\begin {theorem}
The maxWCBS problem is NP-hard.
\end {theorem}
	
\begin {proof}
We use the argument for Theorem \ref{balanced-NPhard}.
Let $G$ be a bipartite graph. 
We construct a complete weighted bipartite graph $G^{*}$ whose edges have weight $1$ if they belong to the graph $G$ and $0$ otherwise.
In the graph $G^{*}$ there is a biclique $x$ of weight $k^{2}$ on $2k$ vertices if and only if $x$ is a $k$-balanced biclique in $G$.
Thus, there is a polynomial reduction of BCBS to maxWCBS.
\end {proof}
	
Similar to balanced case, with each feasible solution of maxWCBS we associate a characteristic vector $x \in \R^{n^{2}}_{+}$.
Let $X^{u}_{n,k}$ be the set of characteristic vectors of all feasible solutions.
We consider the nonnegative cone decomposition $K^{max}_{n,k}$ of the positive orthant $\R^{n^{2}}_{+}$ by the set $X^{u}_{n,k}$ for the maximum problem (\ref {ConeMax}).
	
\begin{theorem}
Clique number of the graph of the cone decomposition $K^{max}_{n,k}$ for the maximum weighted complete bipartite subgraph problem is superpolynomial in $n$ and $k$:
\begin{gather*}
\omega (K^{max}_{n,k}) \geq {n \choose s} = \Omega \left(\left(\frac{n}{s}\right)^{s}\right), \mbox { if } k = 2s,\\
\omega (K^{max}_{n,k}) \geq {n-1 \choose s} = \Omega \left(\left(\frac{n-1}{s}\right)^{s}\right), \mbox { if } k = 2s+1.
\end{gather*}
\end{theorem}
	
\begin{proof}
If $k$ is even, we can use the construction from the proof of Theorem \ref{balanced-clique} where the vertices with the same numbers are taken from the left and right parts.
The nonnegative cones of $s$-balanced bicliques without common parts will be pairwise adjacent by Lemma \ref{balanced-adjacency}.
Note that in the proof of this part of the adjacency criterion only nonnegative 0/1-weights were used (Fig. \ref{Fig_1}).
		
The case with odd $k$ is treated similarly.
We fix one vertex in the left and right parts. Let it be the vertex with the number $n$.
We consider the set $Y^{max}_{n,k}$ of bipartite subgraphs of the following form:
\begin{gather*}
|U_{x}| = s,\ U_{x} \subseteq \{1,2,3,\ldots n-1\},\\
|V(x)| = s+1,\ \{n\} \in V(x),
\end{gather*} 
that contain the vertex with the number $n$ in the right part.
		
For any two subgraphs $x$ and $y$ of $Y^{max}_{n,k}$, we can construct a vector $c$ of edge weights by the rule (\ref{0/1-weight}).
The subgraphs $x$ and $y$ have the greatest possible weight for a biclique on $2s + 1$ vertices:
$$\left\langle c,x\right\rangle = \left\langle c,y\right\rangle = s(s+1).$$

Any bipartite subgraph $z$ that includes vertices other than $U (x) \cup U (y) $, or both vertices of $U (x) \backslash U (y)$ and $ U (y) \backslash U (x) $ will have a sum of weights less than this.
Therefore, by the condition (\ref {adjacency-cones}), the cones $K^{max}_{n,k} (x)$ and $K^{max}_{n,k} (y)$ are adjacent, and $Y^{max}_{n,k}$ forms a clique in the graph of the cone decomposition.
If we exclude a vertex with number $n$, the case of odd $k$ reduces to an even case for the graph on $n-1$ vertices in each part.
\end {proof}
		
We note that all the results and proofs for the maximum complete bipartite subgraph problem are similar to the balanced case.
This is expected, given that to maximize the number of edges (edge weights) the required subgraph should be as close to a balanced one as possible.
For the minimum problem, it will be different.

\section {Minimum complete bipartite subgraph}

\begin {theorem}
The minWCBS problem is NP-hard.
\end {theorem}
		
\begin {proof}
We consider a bipartite graph $G$ with edges of two possible weights: $1$ and $n^2$.
Suppose that instead of selecting vertices in a subgraph, we exclude vertices and all edges incident to them from the complete graph.
In this case, the minWCBS problem of finding a bipartite subgraph on $k$ vertices of minimum weight is equivalent to the problem of finding a subset of $2n-k$ vertices with a maximum number of incident edges of weight $n^2$ for elimination.
Note that the total contribution of all edges of weight $1$ is less than the contribution of one edge of weight $n^2$.
We construct an unweighted bipartite graph $G^{*}$ by removing all edges of unit weight from $G$.
The described problem is the maximum $q$-vertex cover problem on graph $G^{*}$ for $q = 2n-k$.
		
\textbf{Maximum $q$-vertex cover.}
Given a graph  $G = (V,E)$ and a positive integer $q < |V|$.
It is required to find a subset of vertices $U \subset V$ with the greatest number of incident edges such that $|U| = q$.
		
The maximum $q$-vertex cover problem is NP-hard even restricted to bipartite graphs \cite {Apollonio, Joret}.
\end {proof}
		
We consider the cone decomposition $K^{min}_{n,k}$ of the space $\R^{n^{2}}_{+}$ by $X^{u}_{n,k}$ where $X^{u}_{n,k}$ is the set of characteristic vectors of all feasible solutions of the minWCBS problem (\ref {ConeMin}).
Denote by:
$$m = \left\lfloor\frac{n}{2} \right\rfloor.$$
		
\begin {theorem}
Let
\begin {equation} \label {kIneq}
k = 3s \mbox { and  } \frac {9}{4} m < k < 3m,
\end {equation}
then the clique number of the graph of the cone decomposition $K^{min}_{n,k}$ for the minimum weighted complete bipartite subgraph problem is superpolynomial in $n$ and $k$:
$$\omega (K^{min}_{n,k}) \geq {m \choose s} = \Omega \left(\left\lfloor \frac{3n}{2k}\right\rfloor^{\frac{k}{3}}\right).$$
\end {theorem}
		
\begin {proof}
Without loss of generality, we assume that $n = 2m$.
Otherwise, we can choose a vertex in each part and exclude them from further construction.
		
We divide the sets of vertices in each part into two equal disjoint subsets $U = U_{1} \cup U_{2}$, $V = V_{1} \cup V_{2}$.
Inside each set, we number the vertices from $1$ to $m$.
		
Consider the set $Y^{min}_{n,k}$ of subgraphs of the following form: for any subset
$$\{x_{1},x_{2},\ldots,x_{s}\} \subset \{1,2\ldots m\},$$
we construct a bipartite subgraph $x$ that includes vertices with numbers $\{x_{1},\ldots,x_{s}\}$ from the sets $U_{1}, V_{1}$, and $V_{2}$.
		
Let us prove that the cones corresponding to subgraphs of $Y^{min}_{n,k}$ are pairwise adjacent.
We consider two arbitrary subgraphs $x,y$ of $Y^{min}_{n,k}$.
Let the function of edge weights $c$ has the following form (Fig. \ref {Fig_3}):
\begin{gather*}
c_{i,j} = 
\begin{cases}
0,& \mbox {if } i \in U(x) \backslash U(y),\ j \in V(x),\\
& \mbox{or } i \in U(y) \backslash U(x),\ j \in V(y),\\
& \mbox{or } i \in U(x) \cup U(y),\ j \in V(x) \cap V(y),\\ 
1,& \mbox {if } i \in U(x) \cap U(y),\ j \in V(x) \triangle V(y),\\
\infty,& \mbox {otherwise.}
\end{cases}
\end{gather*}
		
\begin{figure}[h]
\centering
\includegraphics[width=5in]{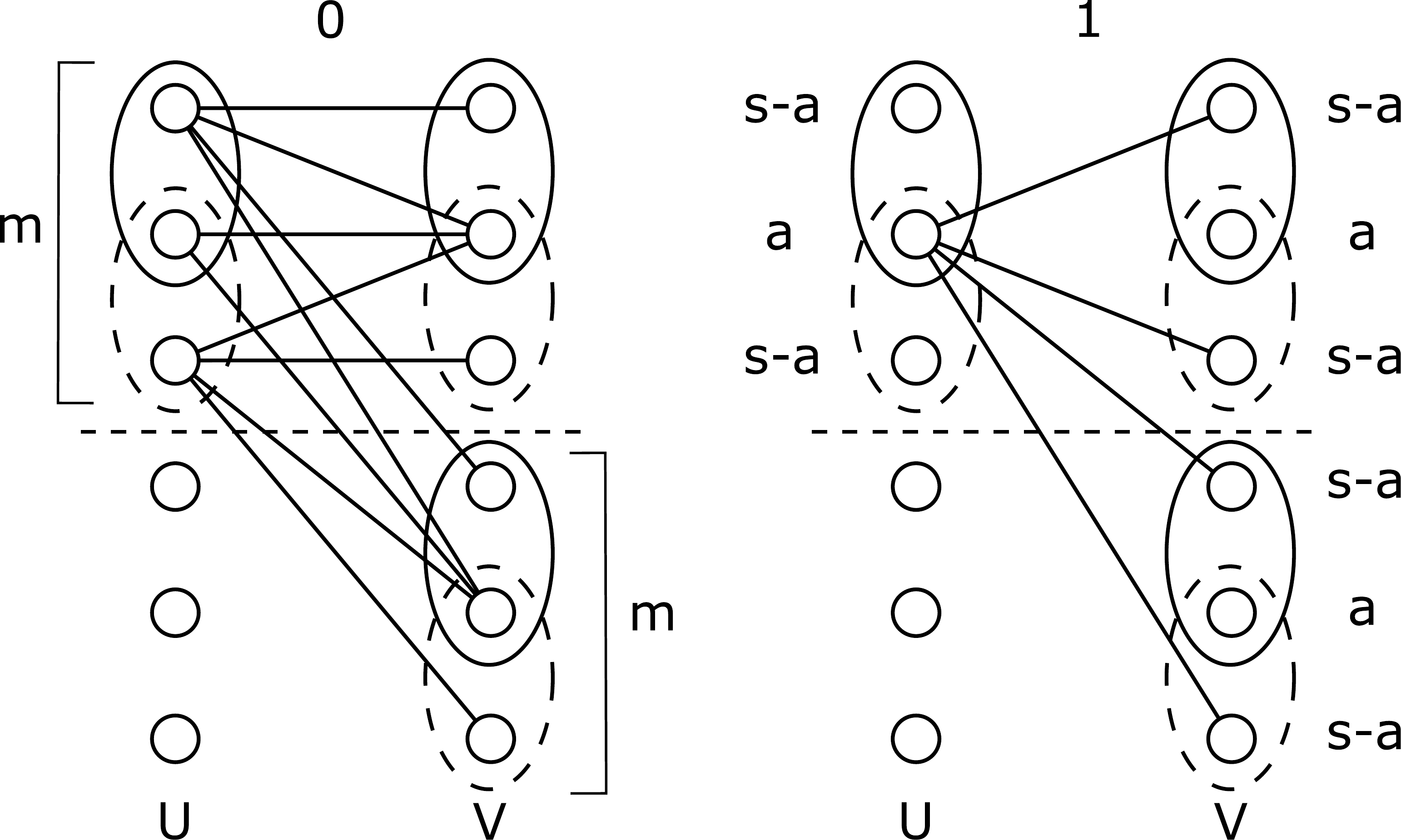}
\caption {The weight function for minimum unbalanced subgraphs}
\label {Fig_3}
\end{figure}
		
Let $|U(x) \cap U(y)| = a$, then
\begin{equation} \label {minXY}
\left\langle c,x\right\rangle = \left\langle c,y\right\rangle = 2 a (s-a).
\end{equation}
		
Note that by construction of the set $Y^{min}_{n,k}$:
\begin{equation} \label {aIneq}
2s-m \leq a < s.
\end{equation}
		
We consider an arbitrary bipartite subgraph $z$ with the number of vertices $k = 3s$.
		
\begin{itemize}
\item If $z$ includes at least one vertex of $U$ that does not belong to $U(x) \cup U(y)$, then $\left\langle c,z\right\rangle=\infty$, since all edges incident to this vertex have infinite weight.
\item If $z$ includes both vertices of $U(x) \backslash U(y)$ and $U(y) \backslash U(x)$, then in the right part $V$ only the vertices of $V(x) \cap V(y)$ are connected with them by edges of finite weight simultaneously:
\begin{gather*}
|V(x) \cap V(y)| = 2a,\\
|U(x) \cup U(y)| = 2s - a,\\
(2s - a) + 2a = 2s + a < 3s = k.
\end{gather*}
In this case, there are not enough vertices for a bipartite subgraph $z$ of finite weight. 
\item If $U(z) \subseteq U(x)$, and $z$ includes at least one vertex of $U(x) \backslash U(y)$, then in the right part only vertices of $V (x)$ have edges of finite weight adjacent to this vertex.
In this case, if $\left\langle c,z\right\rangle < \infty$, then $z = x$.
The case with a vertex from $U(y) \backslash U(x)$ is treated similarly.
			
\item If $U(z) \subseteq U(x) \cap U(y)$, then
\begin{equation} \label{minZ}
\left\langle c,z\right\rangle = b (3s - b - 2a),
\end{equation}
where $b = |U(z)|$.
By construction, $b$ satisfies the following constraints:
\begin{equation} \label{bIneq}
3s - 2m \leq b \leq a.
\end{equation}
Suppose that the weight of the subgraph $z$ (\ref {minZ}) does not exceed the weights of $x$ and $y$ (\ref {minXY}), then, taking into account the inequalities (\ref {kIneq}, \ref {aIneq}, \ref {bIneq}), we have the following system
\begin{gather*}
\begin{cases}
b (3s - b - 2a) \leq 2 a (s-a),\\
\frac {3}{4} m < s < m,\\
2s-m \leq a < s,\\
3s - 2m \leq b \leq a,\\
m,s,a,b > 0,
\end{cases}
\end{gather*}
that has no solutions for any values of the parameters $m,s,a,b$.
\end{itemize}
		
Thus, the weight of any subgraph $z$ is strictly greater than the weights of $x$ and $y$.
By the condition (\ref{adjacency-cones}), the cones $K^{min}_{n,k} (x)$ and $K^{min}_{n,k} (y)$ are adjacent. 
The set $Y^{min}_{n,k}$ forms a clique in the graph of the cone decomposition.
By construction, we obtain
$$|Y^{min}_{n,k}| = {m \choose s} = \Omega \left(\left\lfloor \frac{3n}{2k}\right\rfloor^{\frac{k}{3}}\right).$$
\end{proof}

\section{Conclusion}

The problems of constructing an optimal complete bipartite subgraph have been investigated by many authors and have many practical applications.
For three considered problems of the balanced subgraphs with arbitrary edges and unbalanced subgraphs of minimum and maximum weight with nonnegative edges, NP-completeness of problems and superpolynomial clique numbers of 1-skeletons and graphs of the cone decompositions are established.
In all three cases, the polyhedral characteristics correlate with the complexity of the problems.
	
It should be noted that the minimum complete bipartite subgraph problem is NP-hard, although the close problem of the existence of an unbalanced bipartite subgraph on $k$ vertices is polynomially solvable.
Perhaps in this regard, the minWCBS problem turned out to be the most difficult to analyze, and the complexity of the dual problem of the maximum $q$-vertex cover in a bipartite graph remained unknown for many years \cite {Apollonio}.

\renewcommand{\refname}{References}
{
	}
\medskip
\end{document}